\documentclass[a4paper,11pt]{article}

\usepackage[utf8]{inputenc}   
\usepackage[english]{babel}
\usepackage{verbatim}
\usepackage{amssymb,amsfonts,amsmath,amsthm,qsymbols}
\usepackage{cancel,enumerate}
\usepackage{hyperref}
\usepackage{xcolor,graphicx}

\usepackage[scale=0.75]{geometry}
\def \MEX{{\sf MEX}}
\def \Exp{{\sf Exp}}
\def \Expo{{\sf Expo}}
\def \cro#1{\llbracket#1\rrbracket}
\def \eref#1{(\ref{#1})}
\def \wh#1{\widehat{#1}}
\def \1{{\bf 1}}
\def \l{\left}
\def \r{\right}

\def \ben{\begin{eqnarray}}
\def \een{\end{eqnarray}}
\def \be{\begin{eqnarray*}}
\def \ee{\end{eqnarray*}}

\def \G{{\sf G}}
\def \beq{\begin{equation}}
\def \eq{\end{equation}}

\def \ov#1{\overline{#1}}

\newqsymbol{`E}{\mathbb{E}}
 \newqsymbol{`P}{\mathbb{P}}
\newqsymbol{`R}{\mathbb{R}}
\newqsymbol{`e}{\varepsilon}
\newqsymbol{`o}{\omega}
\def \iuk{{i \in \cro{1,k}}}
\def \izk{{i \in \cro{0,k}}}

\newtheorem{lem}{Lemma}[section]
\newtheorem{pro}[lem]{Proposition}

\newtheorem{theo}[lem]{Theorem}

\theoremstyle{definition}

\newtheorem{rem}[lem]{Remark}

\def\Op{{\sf Op}}

\def \dd{\xrightarrow[n]{(d)}}

\newcommand{\sur}[2]{\mathrel{\mathop{\kern 0pt#1}\limits^{#2}}}
\newcommand{\sous}[2]{\mathrel{\mathop{\kern 0pt#1}\limits_{#2}}}
\newcommand{\eqd}{\sur{=}{(d)}}
\newcommand{\reals}{\mathbb{R}}

\newcommand{\sort}[1]{\text{\sf sort}\left( #1 \right)}
\newcommand{\gaps}[1]{\text{\sf gaps}\left( #1 \right)}

\newcommand{\esp}[1]{\mathbb{E} \left[ #1 \right]}
\newcommand{\prob}[1]{\mathbb{P} \left[ #1 \right]}
\newcommand{\abs}[1]{\left| #1 \right|}

\newcommand{\compact}{
  \topsep0pt
  \itemsep=0pt
  \partopsep=0pt
  \parsep=0pt
}

\title{Processes iterated \it ad libitum.}

\begin{document}
\renewcommand{\baselinestretch}{1.2}
\setcounter{page}{1}

 \begin{center}
 \LARGE\bf
Processes iterated \it ad libitum\rm.\\
 {\large \bf Jérôme Casse and Jean-Fran\c{c}ois Marckert}
 \rm \\
 \large{CNRS, LaBRI \\ Universit\'e Bordeaux \\
  351 cours de la Libération\\
 33405 Talence cedex, France}
 \normalsize
 \end{center}

\begin{abstract} 
Consider the $n$th iterated Brownian motion $I^{(n)}=B_n \circ\cdots \circ B_1$. Curien and Konstantopoulos proved that for any distinct numbers $t_i\neq 0$, $(I^{(n)}(t_1),\dots,I^{(n)}(t_k))$ converges in distribution to a limit $I[k]$ independent of the $t_i$'s, exchangeable, and gave some elements on the limit occupation measure of $I^{(n)}$. Here, we prove under some conditions, finite dimensional distributions of $n$th iterated two-sided stable processes converge, and the same holds the reflected Brownian motions. We give a description of the law of $I[k]$, of the finite dimensional distributions of $I^{(n)}$, as well as those of the iterated reflected Brownian motion iterated ad libitum.
\end{abstract}

\section{Introduction}

\label{sec:int}

Let $B,B_{1},B_2,\dots$ be a family of i.i.d. independent two-sided Brownian motions (BM), meaning that for any $n$, $(B_n(t), t \geq 0)$ and  $(B_n(-t), t \geq 0)$ are two independent standard linear BM. Denote by $I^{(n)}=B_n \circ \cdots \circ B_1$ the $n$th time iterated BM. 
Curien and Konstantopoulos \cite{C-K} obtained the following results, gather in the following proposition.
\begin{pro}\label{pro:C-K} (1) For any $k\geq 1$, any non zero $t_1,\dots,t_k$, $(I^{(n)}(t_1),\cdots,I^{(n)}(t_k))$ converges in distribution. The limit distribution $\mu_k$ does not depend on the $t_i$'s, and then is exchangeable.\\
 (2) For  $(I_1,\dots,I_k) \sim \mu_k$, the equality $(I_1,\dots,I_k)\eqd (B(I_1),\cdots, B(I_k))$ holds.
Moreover, 
\[ (I_2-I_1,\dots, I_k-I_1)\eqd (I_1,\cdots,I_{k-1})\sim \mu_{k-1}.\]
The distribution of $I_1$ possesses the density $\exp(-2|x|)$ over $\mathbb{R}$ (this result appeared first in  Turban \cite{TUR}),\\
(3) Let $\phi_n$ be the occupation measure of $I^{(n)}$ on $[0,1]$, then the sequence $(\phi_n,n \geq 0)$ converges as $n \to \infty$ in distribution to a random probability measure $\phi$, which has a.s. a finite support, and which has a.s. a  H\"older continuous density with exponent $1/2-\epsilon$ for all $\epsilon>0$.
\end{pro} 

In this paper we go on this study in several connected directions: among other we give some elements on $\mu_k$, study iterated reflected BM, iterated stable processes, and provide a description of the finite dimensional distribution of the $n$th iterated BM $I^{(n)}$.
\medskip 

Here are the main lines of the paper. 
In Section \ref{sec:P} we present the studied processes and fix some notations.
In Section \ref{sec:FC} we provide some common features of the processes we iterate. Given a finite set of points $L=\{\ell_i, i=0,\dots,k\}$, the gaps sequence of $L$ is the sequence $G = (\wh{\ell}_i-\wh{\ell}_{i-1},1\leq i \leq k)$, of differences of successive points in $L$ when sorted in increasing order. It turns out that for processes $X$ with independent and stationary increments, the distribution of the gaps sequence of $X(L)=\{X(\ell_i),i\in I\}$ can be described uniquely using the gaps sequence $G$ of $L$. This simple property will appear to be at the heart of our advances about iterated BM. \par
In Section \ref{sec:BP}, devoted to iterated BM and iterated reflected BM, it is explained that if the initial gaps sequence $G$ is a $k$ tuple of independent exponential random variables (r.v.) with parameters $(\lambda_1,\dots,\lambda_k)$ then the gaps sequence of $X(L)$ is distributed according to a mixture of k-tuple of independent exponential r.v., whose parameters are explicit functions of $(\lambda_1,\dots,\lambda_k)$. To encode this property, we define a Markov chain $(Z^{(n)},n\geq 1)$ at the parameter level, which makes explicit this parameter evolution (see \eref{eq:parameter-kernel} and around). \par
A consequence is that the gaps sequence of the iterated BM ad libitum is a mixture of independent exponential r.v., and this mixture can be described precisely using the invariant distribution of the Markov chain $Z^{(n)}$ (Propositions \ref{pro:etsdy}, \ref{pro:ap} and Theorem \ref{theo:multi-dim}). \par
Somehow, Remark \ref{rem:Hutch} implies that our description of the iterated BM finite dimensional distribution, while complex, is the simplest we could expect.\par

The same construction, using an analogous of the parameter Markov chain $(Z^{(n)},n\geq 1)$,
 implies that the law of the $n$th iterated BM is accessible if the gaps sequence of the initial distribution follows some independent exponential r.v. In Section \ref{sec:IRMn}, it is seen that this property provides a Laplace type transform of the finite dimensional distributions of the $n$th iterated BM $I^{(n)}$. 
Section \ref{sec:IRMal} is devoted to the iteration of reflected BM.
Section \ref{sec:ISPal} is devoted to the iteration of stable processes, whose study appear much similar to that of iterated BM, except that explicit computations are out of reach for the moment.

We discuss in Section \ref{sec:Con} some natural extensions of this work.

\section{Random processes} 
\label{sec:P}

``BM'' will be used to denote the two-sided linear BM as defined at the beginning of Section \ref{sec:int}. The process corresponding to the $n$th iterated process will be denoted $I^{(n)}$ (the process iterated under discussion, denoted $X$ further, will be clear from the context). The processes iterated ad libitum, the limit of $I^{(n)}$ in the sense of the topology of finite dimensional distribution convergence, when it exists will be denoted $I$. The reflected BM is the (one-sided) process $(|B(t)|,t\geq 0)$ where $B$ is the standard linear BM. 

We go on discussing stable processes (see Applebaum \cite{AP} for more information). We will consider only two-sided stable variables $Z$ that can be written under the form $A+r$ where $A$ is stable symmetric (null skewness), and $r$ a real number (the location parameter).  The characteristic function of such a r.v. $Z$ can be written under the form 
\[\psi(u)=`E(e^{iuZ})=e^{ \eta(u)}\]
where
\[\eta(u)= -\abs{u}^\alpha \sigma^\alpha +  i  r u \]
where $\alpha \in (0,2]$ is the index of stability, $\sigma \in (0,\infty)$ the scale parameter 
(Theorem 1.2.21 in \cite{AP}). 
A stable process $(X(t),t \geq 0)$ with parameters $(\alpha,\sigma,r)$ is the process such that $ X(0)=0$, with stationary and independent increments, and whose characteristic function is
\begin{equation}
\Phi_t(u) = \esp{e^{i u X(t) }} = e^{ t \eta(u)}.
\end{equation}
The two sided stable process  $(X(t),t \in \mathbb{R})$ is the process such that $(X(t),t\geq 0)$ and $(X(t),t<0)$ are independent and $(X(t),t\geq 0)$ and $(-X(-t),t\geq 0)$ are both one-sided stable process with parameters $(\alpha,\sigma,r)$. 
For any $t\in \mathbb{R}^\star$,
\ben\label{eq:rec}
\frac{X_t-tr}{|t|^{1/\alpha}}\eqd X_1-r.
\een
For any  $c> 0$, $(X(c^\alpha t),t \geq 0)$ is a stable process with parameters $\left( \alpha,c \sigma, c^{\alpha} r \right)$. 

Let  $(X_1,X_2,\dots,)$ be a family of i.i.d. two-sided stable processes with parameters $(\alpha,\sigma,r)$. 
The $n$th iterated stable process $I^{(n)}$ of parameters $(\alpha,\sigma,r)$ is the process 
\begin{displaymath}
I^{(n)}=X_n \circ \cdots \circ X_1.
\end{displaymath}
We keep the same notation as for the iterated two-sided BM for some reasons that will appear clear below.

\begin{rem} \label{rem:bmit}
The BM is the stable process with parameters $(2,1/\sqrt{2},0)$. Its Markov kernel is $`P(B_{t+s}\in dy | B_s=x)=\exp(-(y-x)^2/(2t))/\sqrt{2\pi t}$.
\end{rem}
Iteration of stable processes with parameter $(\alpha,1,0)$ and $(\alpha,\sigma,0)$ can be directly compared as explained in Remark \ref{rem:comp}.

\section{Iteration of processes: general considerations}
\label{sec:FC}

In this section, we discuss some common features of the processes we iterate in the paper. \par

All along the section $k$ is a positive integer: the size of the finite dimensional distributions under inspection.

\paragraph{Notations.} We denote by $\cro{a,b}$ the ordered sequence $[a,b]\cap \mathbb{Z}$. 
The permutation group of the set $\cro{a,b}$ is denoted ${\cal S}\cro{a,b}$. Sometimes, we will use the notation $x[a:b]$ instead of $(x_a,\dots,x_b)$, and also $t[a:b], g[a:b], \lambda[a:b]$, etc, accordingly. 
The simple notation $x[k]$ will stand for $x[1:k]$.\par
For any sequence $\ell[0:k]=(\ell_0,\dots,\ell_k)$, denote by  $(\wh{\ell}[0:k])=\sort{\ell[0:k]}$ this sequence sorted in increasing order.
For any $i\in\cro{1,k}$, set
\[\Delta \ell_i=\ell_i-\ell_{i-1}.\]
The gaps sequence of $\ell[0:k]$ is the sequence of distances between the elements of $\{\ell_0,\cdots,\ell_k\}$. It is defined by
\[\gaps{\ell[0:k] }=\l(\Delta \wh{\ell}_i, i \in\cro{1,k}\r).\]
Last, for $x[1:k]$ a sequence, $\bar{x}[0:k]$ is the sequence defined by
\ben\label{eq:S}\bar{x}_0=0,~\ov{x}_i=x_1+\dots+x_i,~~ \textrm{ for } \iuk.\een

\paragraph{Iteration of processes.}

What follows is valid for processes $X$ such that $X(0)=0$ a.s., with independent and stationary increments, which distribution are absolutely continuous with respect to the Lebesgue measure on $`R$ such that for any $t>s$, 
\ben\label{eq:qds}
X(t)-X(s)\eqd X(t-s)\een
whatever are the signs of $s$ and $t$. Notice that this implies $-X(s)\eqd X(-s)$ (taking $t=0$).
These general setting are satisfied by BM, by symmetric two-sided stable processes, and more generally, by symmetric two-sided Lévy processes such that for any $t>0$, $X(t)$ owns a density.  Some modifications are needed for processes as the reflected BM which have stationary but dependent increments. This is discussed in Section \ref{sec:IRMal}.

Denote by $\Phi_t(.)$ the density of the distribution of $X(t)$. We then have 
\ben \label{eq:sym}
\Phi_{t}(y)=\Phi_{-t}(-y),~~ \textrm{ for any }(t,y) \in \mathbb{R}^\star \times \mathbb{R}.
\een

\subsection{The gaps sequence evolution}
Let $(t_0=0,t_1,\dots,t_k)$ be some distinct real numbers. We start with the description of the distribution of $(X(t_i), i\in\cro{0,k})$. As usual, the description is easier if the $t_i$ are sorted... 

Let $\tau\in{\cal S}\cro{0,k}$ such that $(\wh{t}_i=t_{\tau(i)},i\in\cro{0,k})=\sort{t[0:k]}$. Hence , $\wh{t}_{\tau^{-1}(0)}=0$. Further let $g[k] = \gaps{t[0:k]}$. The r.v. $(X(\wh{t}_{i+1})-X(\wh{t}_{i}), i\in\cro{1,k})$ are independent, and $X(\wh{t}_{i})-X(\wh{t}_{i-1})\eqd X(\Delta \wh{t}_{i})$ depends on the gaps sequence of the $t_i$'s. Using the independence of the increments of $X$ and their stationary, we get that the density $f$ of $(X(t_i), \iuk)$ on $`R^k$ is
\ben\label{eq:evoll}
f(y[k])=
\prod_{j=1}^k \Phi_{\Delta \wh{t}_j}\l(\Delta y_{\tau(j)}\r).
\een
where  in the right hand side $y_0=0$. Indeed, one has  $(X(t_i),\izk)=(X(\hat{t}_{\tau^{-1}(i)}),\izk)$, and computing
$`P\l(X(\hat{t}_{\tau^{-1}(i)} \in dy_i, \iuk \r)= `P\l(X(\hat{t}_{i}) \in dy_{\tau(i)}, \iuk \r)$ gives the result, using \eref{eq:sym}.  

The distribution of $\gaps{X(t_i), \izk}$ depends also only on $\gaps{t[0:k]}$, and this is one of the key point of the paper. First, determine the vectors $(X(t_i), i\in\cro{0,k})$ such that
\ben\label{eq:gaps}
\gaps{X(t_i), i\in\cro{0,k}}=x[k]
\een 
some fixed element of $ (0,+\infty)^k$. Clearly \eref{eq:gaps} holds iff there exists some $a\in\mathbb{R}$ such that
 \ben\label{eq:qdq}
 \sort{X(t_i), i\in\cro{0,k}}=\l( a+\ov{x}_i , i\in\cro{0,k}\r).
 \een
Equation \eref{eq:qdq} implies that for a certain permutation $\tau \in {\cal S}\cro{0,k}$ 
\[(X(\wh{t}_{i}),i\in\cro{0,k})=(a+\ov{x}_{\tau(i)}, i\in\cro{0,k})\] 
from what we find 
 \ben\label{eq:evol}
\l(X(\wh{t}_{i})-X(\wh{t}_{i-1}), i \in\cro{1,k}\r)=(\Delta \ov{x}_{\tau(i)},i \in\cro{1,k}).
\een
The following proposition should be clear now
 \begin{pro}\label{pro:pro2}  Let $t[0:k]$ be $k+1$ distinct real numbers with $t_0=0$ such that
 \[\gaps{t[0:k]}=g[k]\in (0,+\infty)^k.\]
 The distribution of  $\gaps{(X(t_0),\dots,X(t_{k})}$ has density $\Psi_{g[k]}$ 
on $(\mathbb{R}^+)^k$ where
\ben \label{eq:Psi}
\Psi_{g[k]}(x[k])=\sum_{\tau\in {\cal S}\cro{0,k}} \prod_{i=1}^k \Phi_{g_i}\l( \Delta \ov{x}_{\tau(i)}\r) 
\,1_{x_i>0}.
\een
\end{pro}
In the mono-dimensional case,  
\ben\label{eq:double}
\Psi_g(x)= (\Phi_g(x)+\Phi_g(-x))1_{x\geq 0}
\een 
and this is also $2\Phi_g(x)1_{x\geq 0}$ when $\Phi_g$ is even (that is when $r=0$ in the stable processes case). 

We may now define a time-homogeneous MC $(\G^{(n)}[k]=(\G_i^{(n)}, i \in\cro{1,k}), n\geq 0)$  taking its values in $(0,+\infty)^k$, giving the successive gaps sequence starting from an initial one; its Markov kernel is given by $\Psi$ in the sense of Proposition \ref{pro:pro2}. We will call $\G^{(n)}$ the gaps sequence MC. 
Assume that $\G^{(0)}_k$  is a r.v. which possesses a density $f_k$ on $(0,+\infty)^k$. The density of $\G^{(1)}_k$  is $\Op_k(f_k)$ where $\Op_k$ is the following integral operator (which sends $f_k$ onto $\Op_k(f_k)$), where for any $x[k]\in `R^k$,
\ben\label{eq:Opk}
  \Op_k(f_k)(x[k]):=\int\cdots\int f_k(g[k]) \Psi_{g[k]}(x[k]) dg_1\dots d{g_k}.
\een

Of course, if one considers a case for which the iterated process converges in distribution,
\[I^{(n)}[k]=(I^{(n)}(t_1),\dots,I^{(n)}(t_k))\dd I[k]=(I(t_1),\dots,I(t_k))\] 
then the associated gap MC $(\G^{(n)}[k], n\geq 0)$ converges too  since the map $x[0;k]\to \gaps{x[0:k]}$ is continuous. 
The converse is false but not that much: the gaps sequence characterises the  points relative positions. An additional information is needed to recover their positions: somehow the distribution of the translations which sends $\gaps{I(t_i), \izk}$ onto $\{I(t_i), \izk\}$, and the distribution of the permutation which provides the distribution of $(I(t_i), \izk)$ knowing $\{I(t_i), \izk\}$. A simple but powerful trick,  discussed at several places in the paper is the following : we are able to pass from the gaps sequence MC to the usual one if instead of $(I^{(n)}(t_i), \iuk)$, we study $(I^{(n)}(t_i), \izk)$  instead, where $t_0=0$.  

We can sum up in two slogans the relative importance of the iteration of the initial process $X$ with respect to the gap MC: the proof of convergence is easier for $(I^{(n)}(t_i),\izk)$, but the behaviour of $\G^{(n)}[k]$  is easier to understand, and its distribution in the case of Brownian processes is tractable.

\subsection{The iterated process evolution}
\label{sec:IPE}
Any sequence  $t[0:k]$ such that $t_0=0$ can be encoded by the pair  $C[t]:=\l[g[k],\tau\r]$ formed by the gaps sequence of $t$, and the ``labelling permutation'' $\tau \in {\cal S}\cro{0,k}$,  so that
\ben\label{eq:enc}
 t_i= \ov{g}_{\tau(i)}-\ov{g}_{\tau(0)}, \izk.
\een
Of course, thanks to \eref{eq:enc}, the decoding $t=C^{-1}(g[k],\tau)$ is well defined too (taken $t_0=0$).
Follows from \eref{eq:enc} again, that $t_{\tau^{-1}(i)}$is non decreasing in $i$ and then for any $i$ we have
\ben\label{eq:whtau}
\wh{t}_i=t_{\tau^{-1}(i)}=\ov{g}_{i}-\ov{g}_{\tau(0)}.
\een
The Markov kernel of the MC $n\mapsto(I^{(n)}(t_i), \izk)$ can be made explicit at the level of the encodings.  Consider $t[0:k]$ with $t_0=0$ and $C[t]:=\l[g[k],\tau\r]$ its encoding, and $t'[0:k]$ with $t'_0=0$  and $C[t']:=\l[g'[k],\tau'\r]$ its encoding. Denote by $K$ the corresponding Markov kernel (with transparent convention) which gives the distribution of $C[I^{(n+1)}]$ knowing  $C[I^{(n)}]$. We have
\be
\l\{X(t_i) \in dt'_i, \iuk\r\}&=&\l\{X(\wh{t}_i) \in d t'_{\tau^{-1}(i)}, \iuk\r\}
\ee
 and then using \eref{eq:evol}, \eref{eq:whtau} and  $t'_i= \ov{g'}_{\tau'(i)}-\ov{g'}_{\tau'(0)}$ for $\izk$,
\be
K_{g[k],\tau}[(dg'_1,\dots,dg'_k),\tau']&=&`P(X(t_i) \in dt'_i, \iuk)\\
                                  & = & \prod_{i=1}^k \Phi_{\Delta \wh{t}_i}\l(\Delta  t'_{\tau^{-1}(i)}\r)\\
                                  & = & \prod_{i=1}^k \Phi_{g_i}\l(\Delta \ov{g'}_{\tau'(\tau^{-1}(i))}\r). 
\ee
We rewrite more simply as
\ben\label{eq:tran}
K_{g[k],\tau}[(dg'_1,\dots,dg'_k),\tau'\circ \tau]=\prod_{i=1}^k \Phi_{g_i}\l(\Delta \ov{g'}_{\tau'(i)}\r) 
\een
from what we observe that 
\ben\label{eq:recws}
K_{g[k],\tau}[(dg'_1,\dots,dg'_k),\tau'\circ\tau]=K_{g[k],Id}[(dg'_1,\dots,dg'_k),\tau']\een
and then the LHS is independent of $\tau$. Of course, all of this is valid for $\tau,\tau' \in {\cal S}\cro{0,k}$, and for positive $g_i$'s, $g'_i$'s.

\subsection{Asymptotic independence of labelling permutation and gaps sequence}
\label{eq:AIL}
We explain now why in the encoding Markov chain $(C[I^{(n)}[k],n\geq 1)$, the gaps sequence ``becomes progressively'' independent from the labelling permutation as stated in the main convergence theorems of the paper, where this appears under the form of exchangeability of the limiting distribution $\gamma_k$. The asymptotic exchangeability can be proved directly (see \cite{C-K} or the end of Section \ref{sec:PT}). It is somehow quite complex since it relies on the convergence of $(I^{(n)}(t_1),\dots,I^{(n)}(t_k))$ to a limit independent of the $t_i$'s, and the proof relies on some (classical but) involved estimates. 

We present here another argument which makes this more apparent and which we think, can be of some interest if ones tries to iterate some processes for which the arguments developed in Section \ref{sec:PT} fail.

It is a coupling argument. For a fixed pair $(g[k],\tau)$, consider (using \eref{eq:recws}),
\be
\underline{K}_{g[k],\tau}[(dg'_1,\dots,dg'_k),\tau'']&=&\min_{\tau'}K_{g[k],\tau}[(dg'_1,\dots,dg'_k),\tau'\circ \tau]\\
&=&\min_{\tau'}K_{g[k],Id}[(dg'_1,\dots,dg'_k),\tau']
\ee
the ``minimal flow'' going to $[(dg'_1,\dots,dg'_k),\tau']$ from $(g[k],\tau)$, minimum taken on the $\tau'\in {\cal S}\cro{0,k}$.

In general $\underline{K}$ is a defective Markov kernel. Since it does not depend on $\tau''$, the marginal restriction of $\underline{K}$ to the permutation labelling, is the uniform distribution on ${\cal S}\cro{0,k}$.

Therefore, $\underline{K}_{g[k],\tau}[(dg'_1,\dots,dg'_k),\tau'']$ possesses a simpler form:
\ben
\underline{K}_{g[k],\tau}[(dg'_1,\dots,dg'_k),\tau'']=\kappa_{g[k]}(dg'_1,\dots,dg'_k) \frac{1_{\tau''\in {\cal S}\cro{0,k}}}{(k+1)!}
\een
where $\kappa$ is a defective Markov kernel on $`R^+{}^k$. Let 
\be
q({g[k]}) & = & \kappa_{(g[k])}(`R^+{}^{k})
\ee 
be the total mass of $\underline{K}_{g[k],\tau}$ and of $\kappa_{g[k]}$. (Notice that the cases treated in the present paper, for $g[k]\in(0,\infty)^k$, $q({g[k]})>0$.)  Now set 
\begin{equation}\left
\{\begin{array}{rcl}
K^{[1]}_{g[k],\tau}[(dg'_1,\dots,dg'_k),\tau'\circ\tau]&=&\frac{\kappa_{g[k]}(dg'_1,\dots,dg'_k)/{(k+1)!}}{q({g[k]})}\\
K^{[2]}_{g[k],\tau}&=&\frac{K_{g[k],\tau}-q({g[k]})K^{[1]}_{g[k],\tau}}{1-q({g[k]})}
\end{array}\right.
\end{equation}
so that $K^{[2]}$ is indeed a Markov kernel. It is easily seen that the initial kernel $K$ can be represented as 
\[K_{g[k],\tau}=q({g[k]})K^{[1]}_{g[k],\tau}+(1-q({g[k]}))K^{[2]}_{g[k],\tau},\]
which is the core of our coupling: to sample the MC $C[I^{(n)}]$ from $(g[k],\tau)$, first, sample a Bernoulli random variable with parameter $q({g[k]})$. If it is 1, then use the kernel $K^{[1]}$, else the kernel $K^{[2]}$. 
If the kernel $K^{[1]}$ is used, the new value $(G[k+1],\tau_{k+1})$ has the following property: $\tau_{k+1}$ is uniform and independent from $G[k+1]$, which has distribution  $\kappa_{g[k]}(.)/q({g[k]})$. 

Then as soon as a transition $K^{[1]}$ is used the labelling permutation and the gaps sequence become independent, and this independence carry on, since by $K$ the labelling permutation evolves somehow independently from the current labelling permutation (and it evolves by product, see \eref{eq:recws}). 
It remains to say some words about the frequency of these renewal events: letting $C[I^{(n)}]=(G[k]^{(n)},\tau_n)$ the successive values of the encoding chain, one sees that each time the  renewal probability is $q({G[k]^{(n)}})$. To get renewal with probability one in the sequence $C[I^{(n)}]$ we need not much: continuity and positivity of the kernel on each compact, and  tightness of the sequence $C[I^{(n)}]$.

\section{Iteration of Brownian processes}
\label{sec:BP}

This section is devoted to our results concerning the iterated BM ad libitum, iterated reflected BM ad libitum, and $n$th iterated BM. We will consider iteration of standard linear Brownian motion, but using Remark \ref{rem:comp}  iteration of Brownian motions multiplied by a constant can be studied as well.

We start with a key point relative to the description of the Markov kernel of the gaps sequence MC when $X$ is a BM (but many of what follows is valid for more general Gaussian processes). 
In this section, $\Phi_g$ is the density of the centred Gaussian distribution with variance $g$.
We denote further by $\Exp[\lambda,x] = \lambda e^{-\lambda x} 1_{x \geq 0}$ the density of $\Expo[\lambda]$, the exponential distribution with parameter $\lambda$. Let $\MEX_k$ be the set of probability measures on $`R^k$ having a density of the form 
\ben\label{eq:par-law2}
f\l(x[k]\r)=\int_{`R^+{}^k} \l(\prod_{i=1}^k \Exp\l[ \lambda_i,x_i\r]\r)  d\mu(\lambda_1,\cdots,\lambda_k),~~ x[k]\in  `R^k
\een
where $\mu$ is a general probability distribution on $`R^+{}^k$, called the parameter law of $f$. In other words, the set $\MEX_k$ is the set of mixtures of product of exponential distributions. The key result in this section, valid only in the Gaussian case, is the following proposition.
\begin{pro}\label{eq:lin} For any $k\geq 1$,  $\Op_k$ is linear on $\MEX_k$, and then $\MEX_k$ is stable by $\Op_k$. 
\end{pro}
\begin{proof}
We start by the one-dimensional case for which  \eref{eq:double} holds. \par
Let $f_1(x)=\Exp[\lambda,x]$, and let us find $Op_1(f_1)(x)$ 
by computing its Fourier transform
\[FT_0(a)=\int_{x\geq 0} e^{i a x}\int_{g>0} 2\Phi_g(x)\lambda e^{-\lambda g}dg dx.\] 
This is done in two steps: $Op_1(f_1)$ is the density of a positive r.v. $Z$. 
Hence 
\[FT_1(a)= \frac{1}{2}\int_{-\infty}^{+\infty} e^{i a x}\int_{g>0}  2\Phi_g(x)\lambda e^{-\lambda g}dg dx,\] 
is the Fourier transform of $`e Z$ where $`e$ is a uniform random sign, independent of $Z$. By Fubini, one finds that it is $\int_{g \geq 0}\lambda e^{-\lambda x} e^{-ga^2} dg= \frac{1}{1+\frac{a^2}{2\lambda}}$, which is the Fourier transform of $`e Y$ where $Y$ has distribution $\Expo\l[\sqrt{2\lambda}\r]$.  We deduce from that the identity
\ben\label{eq:sim-trans}
\int_0^{+\infty}\Exp[\lambda,x]\Psi_g(x)dg =\Exp[\sqrt{2\lambda},x],~~ x>0.
\een
In words, $\Op_1$ sends $x\mapsto\Exp[\lambda,x]$ on $x\mapsto\Exp[\sqrt{2\lambda},x]$.
\begin{rem}\label{rem:1} Notice that this implies that $\Exp[2]$ is stable by $\Op_1$. This is the result by Curien-Konstantopoulos \cite{C-K} who proved that $I_1\sim `e Y$ where $Y\sim \Expo[2]$.
 \end{rem}

Assume $k\geq 1$ now. Observe the effect of $\Op_k$ on a product of exponential distributions.
By \eref{eq:sim-trans} and \eref{eq:double}, one has for any $x[k]\in(0,+\infty)^k$, any $\tau\in{\cal S}\cro{0,k}$, the identity
\ben\label{eq:elt}
\sum_{\tau\in{\cal S}\cro{0,k}}\int_{`R^+{}^k}\prod_{i=1}^{k} \Big(\Exp[c_i,g_i]
\Phi_{g_i}\l(\Delta \ov{x}_{\tau(i)}\r)\Big)dg_1...dg_k= \sum_{\tau\in{\cal S}\cro{0,k}}
\frac{1}{2^k} \prod_{i=1}^k \Exp\l[\sqrt{2c_i},\l|\Delta \ov{x}_{\tau(i)}\r|
\r].
\een
An important fact appears here, fact valid only in the Brownian case : one can separate the variables $x_i$'s in the right hand side and let appear a product of independent $\Expo[c'_i]$ r.v., thanks to the two following identities
\ben\label{eq:iden1}
\left\{ \begin{array}{rcl}
\Exp[c,x+x']&=&\Exp[c,x]\,\Exp[c,x']/c,\\ 
\Exp[c,x]\,\Exp[c',x]&=&\Exp[c+c',x]\frac{cc'}{c+c'}.
\end{array}\right.
\een
Let us separate the variables, and for this, collect in $E_{\tau,i}$ the contribution relative to $\Exp[.,x_i]$. 

Since $\ov{x}_{\tau(j)}=x_1+\dots+x_{\tau(j)}$, then $|\Delta \ov{x}_{\tau(j)}|= x_{1+\min(\tau(j),\tau({j-1}))}+\dots+ x_{\max(\tau(j),\tau({j-1}))}$.  Let $E_{\tau,i}= \{j: x_i\in |\Delta \ov{x}_{\tau(j)}|\} = \{j: \min(\tau(j),\tau({j+1})) < i \leq \max(\tau(j),\tau({j+1}))\}$ be the sequence of indices $j$ such that $x_i$ appear in $|\Delta  \ov{x}_{\tau(j)}| $.  Further, let 
\[w_{\tau}(c[k])= \frac{1}{2^k} \prod_{i=1}^k \frac{\sqrt{2c_i}}{ F_{\tau,i}(c) }\] 
and $F_\tau(c[k])=(F_{\tau,i}(c[k]),i \in\cro{1,k})$ where
\begin{eqnarray}\label{eq:FW}
F_{\tau,i}(c[k]) &=&\sum_{j \in E_{\tau,i}} \sqrt{2 c_j}.
\end{eqnarray}
As a consequence of the previous discussion, 
\begin{lem}\label{lem:transfo} If  $f$ is the map $f(x[k])=\prod_{i=1}^k \Exp[\lambda_i,x_i]$ for some fixed $\lambda[k]\in(0,+\infty)^k$, then 
  \begin{displaymath}
    \Op_k(f)(x[k]) = \sum_{\tau\in {\cal S}\cro{0,k}} w_{\tau}(\lambda[k]) \prod_{i=1}^k \Exp\l[ F_{\tau,i}(\lambda),x_i\r].  \end{displaymath}
\end{lem}
Of course, this ends the proof of  Proposition \ref{eq:lin}. 
\end{proof}
In the $2$-dimensional case, the 6 functions $F_\tau$ and weights are the following
\begin{eqnarray}
F_{(0,1,2)}(c_1,c_2) = (s_1,s_2) & , & w_{(0,1,2)}(c_1,c_2) = {1}/{4}, \\
F_{(0,2,1)}(c_1,c_2) = (s_1,s_1+s_2) & ,& w_{(0,2,1)}(c_1,c_2) =  1/4 \ \ s_2/(s_1+s_2) , \\
F_{(1,0,2)}(c_1,c_2) = (s_1+s_2,s_2) & , & w_{(1,0,2)}(c_1,c_2) = 1/4 \ \   s_1/(s_1+s_2), \\
F_{(1,2,0)}(c_1,c_2) = (s_2,s_1+s_2) & , & w_{(1,2,0)}(c_1,c_2) = 1/4 \ \ s_1/(s_1+s_2), \\
F_{(2,0,1)}(c_1,c_2) = (s_1+s_2,s_1) & , & w_{(2,0,1)}(c_1,c_2) = 1/4 \ \ s_2/(s_1+s_2), \\
F_{(2,1,0)}(c_1,c_2) = (s_2,s_1) & , & w_{(2,1,0)}(c_1,c_2) = {1}/{4},
\end{eqnarray}
where for short, we have written $s_{i}$ instead of$\sqrt{2c_i}$. 
We now pass to the consequences in terms of iterated BM ad libitum, reflected BM, and in the case of iterated BM.

\subsection{Iteration of BM ad libitum }
\label{sec:BMal}

Proposition \ref{pro:C-K}, ensures the convergence of $(I^{(n)}(t_i), \iuk)$ to $I[k]$ for any distinct and non zeros $t_i$'s, as well as the exchangeability of the limit. Hence $\gaps{I[k]}$ is the limit of the gap MC, and the limit of this MC does not depend on the $t_i$'s. 
The gaps sequence is not sufficient to describe $I[k]$ even up to a permutation (which would be uniform by exchangeability), since the gaps sequence determines the set of elements of the sequence up to a translation. We present a simple trick which allows one to pass this (apparent) difficulty. 

Consider $I[k+1]$ a $\mu_{k+1}$ distributed sequence. Take $U$ a r.v. uniform in $\cro{1,k+1}$ independent from the $I_i$'s. By  Proposition \ref{pro:C-K}
\ben\label{eq:dsqgr}
J[k+1]:=(I_i-I_U, i \in \cro{1,k+1})
\een
is a random sequence with one zero entry (with uniform position), and the rest of its entries has the same distribution as $(I_i,1\leq i \leq k)$. 
Moreover, $\gaps{J[k+1]}=\gaps{I[k+1]}$ since translations conserve gaps sequence.

Denote by $\gamma_k$ be the distribution of $\gaps{I[k+1]}$. The following proposition, consequence of the previous discussion, allows one to get $\mu_k$ using $\gamma_k$. 
\begin{pro} \label{pro:gaps sequenceToISP}
Consider $(G_i,\iuk)$ a random vector distributed according to $\gamma_k$,  $U$ a uniform r.v. on $\cro{0,k}$ and  $\tau$ a uniform random permutation taken in ${\cal S}\cro{0,k}$, all these r.v. being independent.
The following identity holds $(\overline{G}_{\tau(i)}-\overline{G}_U, 0 \leq i \leq k)\eqd J[k+1].$
\end{pro}
It remains to describe $\gamma_k$. The case $k=1$ is a consequence of Proposition \ref{pro:C-K}, see also Remark \ref{rem:1}. 
For $k\geq 2$, this can be obtained by looking at the limit of the gap MC, starting with some initial positive gaps sequence $g[k]$ since, the gap MC inherits from the initial chain (the iterated BM) the property to possess a limiting distribution, independent from the starting point. Clearly, the initial sequence can be taken random (with values in $`R^{+}{}^k$), for example, one can start with some independent exponential r.v. with parameters $\lambda_1,\dots,\lambda_k$... and this is what we will do since Proposition \ref{eq:lin} and Lemma \ref{lem:transfo} allows one to control exactly the evolution of the distribution of the gaps sequence MC in this case. 

Finding the limiting distribution in this case amounts to finding the fixed point of $\Op_k$.

From  Proposition \ref{eq:lin} and Lemma \ref{lem:transfo} one sees that $\Op_k$ sends an element of $\MEX_k$ on a weighted sums of elements of $\MEX_k$, where the total weight is 1: it is a Markov kernel. It can be better understood if instead of seeing the action of $\Op_k$ at the level of functions, it is seen at the level of the parameters (the parameters of the involved exponential distributions): consider a (discrete time homogeneous) MC $(Z^{(n)}[k]=(Z^{(n)}(1),\dots,Z^{(n)}(k)), n \geq 0)$
 defined on $\mathbb{R}^{\star}{}^k$ whose kernel $Q$ is defined, for any Borelian $A$ of $\mathbb{R}^k$, and $\lambda[k] \in \mathbb{R}^{\star}{}^k$ by
\ben\label{eq:parameter-kernel}
Q ( \lambda[k],A)=\sum_{\tau\in {\cal S}\cro{0,k}} w_{\tau[k]}(\lambda) \delta_{F_{\tau}(\lambda[k])}(A). \een
In other words, 
\ben\label{eq:MC-parameter}
`P\l( Z^{(n+1)}[k]=F_{\tau}(\lambda[k])~|~ Z^{(n)}[k]=\lambda[k]\r))=  w_\tau(\lambda[k])  \textrm{ for any }\tau\in {\cal S}\cro{0,k}.\een 
If $Z^{(n)}[k]\sim \nu$, denote by $\nu Q$ the distribution of $Z^{(n+1)}[k]$. We can sum up the preceding consideration as follows:
\begin{pro}\label{pro:etsdy} Assume that the gaps sequence $(\G^{(n)}(i),\iuk)$ at time $n=0$ is a mixture of exponential distributions with density $g_0$ and parameter law $\nu^{(0)}$, then $\nu^{(0)} Q$ is the parameter law of $\Op_k(f_0)$. More generally $\nu^{(0)} Q^n$ is the parameter law of $f_n=\Op_k^{{(n)}}$, the density of  $(\G^{(n)}(i),\iuk)$. 
\end{pro}

We now conclude by discussing the asymptotic behaviour of the parameter law MC.
\begin{pro} \label{pro:ap}
(1) $Q$ is ergodic in $[2,2k^2]^k$ (meaning that for any $\nu_k^{(0)}$ having its support in  $[2,2k^2]^k$,  $\nu_k^{(0)}Q^n$ converges weakly when $n\to +\infty$ to a distribution $\nu_k$ independent from $\nu_k^{(0)}$). \\ 
(2) The probability density $g$ whose parameter law is $\nu_k$ is solution to $\Op_k(g)=g$.
\end{pro}
\begin{proof} 
We prove the two statements. Let ${\cal M}(S)$ be the set of probability measures with support in $S$.
Since the compact set $[2,2k^2]^k$ is stable by any $F_{\tau}$, then  ${\cal M}([2,2k^2]^k)$ is stable by $Q$. 
Take $\nu^{(0)}_k$  in ${\cal M}([2,2k^2]^k)$, and being the parameter law of some function $f_0$. 
Hence, the sequence $(\nu^{(n)}_k:=\nu^{(0)}_kQ^{n},n\geq 0)$  possesses an accumulation point $\nu_k$ in the compact ${\cal M}([2,2k^2]^k)$. Consider a converging subsequence, still denoted $\nu^{(n)}_k$. 
 Recall that $\nu^{(n)}_k$ is the parameter law of $f_n:=\Op_k^{{(n)}} (f_0)$.  
For any fixed $x[k]$, the map $\lambda[k]\to \prod_{i=1}^k\Exp\l[ \lambda_i,x_i\r]$ is bounded continuous on $`R^+{}^k$, therefore $\nu^{(n)}_k\to \nu_k$ implies that for any fixed $x[k]\in(0,+\infty)^k$,
\ben\label{eq:schqfz}
f_n(x[k])=\int \prod_{i=1}^k \Exp\l[ \lambda_i,x_i\r] d\nu^{(n)}_k(\lambda[k])\to f(x[k]):=\int \prod_{i=1}^k \Exp\l[ \lambda_i,x_i\r] d\nu_k(\lambda[k]).\een 
The fact that $f$ is a density can be checked by Fubini. 
Denote by $\eta_n$ the distribution on $`R^k$ whose density is $f_n$ and by $\eta$ the one whose density is $f$. 
By Scheffé's theorem, the simple convergence \eref{eq:schqfz} implies the convergence of $\eta_n$ to $\eta$.

This implies $\eta=\gamma_k$ (by uniqueness of the limit of the gaps sequence Markov chain), and then $f$ coincides with 
$\lim_n \Op_k^{(n)}(f_0)$. By Proposition \ref{pro:etsdy}, $f$ is the density  of $\gamma_k$.
We must add that a function $f$ in $\MEX$ possesses a unique parameter law, which implies that $\nu_k^{(0)}Q^{n}$ possesses a unique accumulation point, and then converges in distribution. The uniqueness of the parameter law comes from \eref{eq:par-law2}, where one sees that if $\nu$ is the parameter law of $f$, then $f$ is the Laplace transform of the measure $\l(\prod_{i=1}^k \lambda_i\r) \nu(\lambda[k])$.
\end{proof}

\begin{rem}
Take a bounded continuous function $f:`R^k\to `R$. Our representation of the gaps sequence distribution of IBM permits to calculate $`E({f(G[k])})$ under $\gamma_k$ and to give a representation using $\nu_k$ only:
\ben \label{eq:efg}
`E({f(G[k])}) & =& \int_{\reals_+^k} f(x[k]) d \gamma_k(x[k]) \\
& = &\int_{\reals_+^k} \l(\int_{\reals_+^k} f(x[k]) \prod_{j=1}^k \lambda_j e^{-\lambda_j x_j} dx_j   \r) d \nu_k(\lambda[k])
\een
hence, it appears clearly that $`E({f(G[k])})$ can be computed thanks to the parameter distribution $\nu_k$ only.
More generally, using Proposition \ref{pro:gaps sequenceToISP}, one can use this formula to compute $`E({f(I[k])})$ to, which can then also be expressed in terms of $\nu_k$ only. \end{rem}

MCs with kernel such as $Q$, that is, which relies on successive applications of a functions $F_\tau$, where $F_\tau$ is taken at random in a set of functions ${\cal F}=(F_\tau,\tau \in {\cal S}\cro{0,k})$ depending (or not) of the current position, are called iterated function system (IFS) in the literature~\cite{BD85,BDEG88,F04}.

Here since $\Theta^{(0)}_k:=[2,2k^2]^k$ is stable by all the $F_\tau$ (for $\tau \in {\cal S}\cro{0,k}$), it is easily seen that for 
\[\Theta^{(n)}_k:=\bigcup_{\tau\in {\cal S}\cro{0,k}} F_\tau(\Theta_k^{(n-1)}),\] 
the sequence $(\Theta_k^{(n)},n\geq 0)$  is a sequence of non increasing compact sets whose (non empty) limit is a compact $\Theta_k$.
Using the portmanteau theorem and the fact that $n\mapsto \Theta^{(n)}_k$ is decreasing for the inclusion partial order (see Figure  \ref{fig:IFS} for a representation of $\Theta_2$) we can establish that for any $k\geq 1$,
$\Theta_k\supset {\sf{Support}}(\nu_k)$. 
\begin{figure}
\centerline{\includegraphics[width = 6cm]{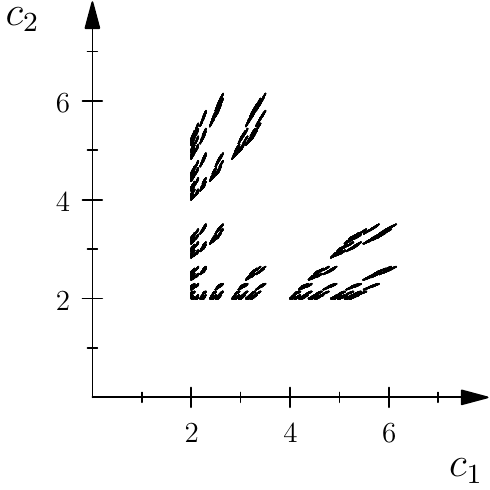}}
\caption{\label{fig:IFS}The support of $\nu_2$ computed by a program.}
\label{fig:support}
\end{figure}
\begin{lem}
$\Theta_2={\sf support}(\nu_2)$.
\end{lem}
\begin{proof}
For $k=2$, it is easily seen that all the $F_\tau$ (given in \eref{eq:FW}) are contracting in $`R^2$ equipped with the Euclidean distance. 
Following classical theorems (e.g. Hutchinson \cite[section 3]{H81}) it turns out that $\Theta_k^{(n)}$ converges to $\Theta_k$ for the Hausdorff metric for any starting set $\Theta0 \subset [2,8]^2$ (and not only from $[2,8]^2$ as stated above). In particular, imagine that $\Theta^{(0)}_2=\{(2,2)\}$, and that the starting measure is $\nu^{(0)}=\delta_{(2,2)}$. Recall that $\nu^{(n)}\to \nu_2$ (since the convergence of  $\nu^{(n)}\to \nu_2$ holds for any starting distribution $\nu^{(0)}$ whose marginals own no atom at $0$). \par
Take any $x\in \Theta_2$, any $`e >0$. By Hutchinson's result, for $n$ large enough $\Theta^{(n)}_2\cap B(x,`e)\neq \varnothing$, which means, taken into account the positivity of the $w_{\tau}'s$, that $ \nu^{(0)}Q^{n} (B(x,`e))>0$: some mass has been transported in a neighbourhood of $x$ in $n$ steps from ${(2,2)}$. 
This is a first step in our proof that $x\in {\sf Support}(\Theta_2)$.
Now, observe that for any $\rho>0$, there exists $m\geq 1$ such that
\[F_{0,1,2}^{\circ m}([2,8]^2)\subset B((2,2),\rho),\]
implying that $m$ iterations of $F_{0,1,2}$ (see \eref{eq:FW}), bring back all the mass (that is 1) in a neighbourhood of $(2,2)$. The probability to proceed to these iterations of $F_{0,1,2}$ is positive (since $\inf_{(c_1,c_2)\in [2,8]^2} w_{0,1,2}(c_1,c_2) >0$). Now, since all the functions $F_{\tau}$ are uniformly continuous on the compact $[2,8]$, for $\rho$ small enough, any distribution $\nu^{(0)}{}'$ with support included in $ B((2,2),\rho)$ will also satisfy $ \nu^{(0)}{}'Q^{n} (B(x,2`e))>0$. \end{proof}
\begin{rem}\label{rem:Hutch}
Hutchinson \cite[section 3]{H81} characterises the set $\Theta_2$: it is the closure of the set of fixed points of the functions $(F_{\tau_1}\circ \dots \circ F_{\tau_m}, m\geq 1, \tau_i \in {\cal S}\cro{0,2})$.
\end{rem}

Using \cite{BDEG88} and some analysis, for $k=3$, the IFS with place-dependent probabilities is contracting in average for the  $\|.\|_2$ distance. This can be proved by computing the Jacobian matrices $J_\tau(c[k])=(\frac{\partial F_{\tau,i}(c[k])}{\partial c_j})_{1\leq i \leq k}$ of the $F_\tau$'s, and by proving that their norms $N_\tau(c[k]):=\sup_{\rho\neq 0} \|\rho J_\tau(c[k])\|_2/\|\rho\|_2$ satisfies
 $\sum_\tau w_\tau(c[k]) \log (N_\tau(c[k]))<0$ (this can be proved by taking first some bounds on the $w_\tau$, and then using the $\log (N_\tau(c[k]))\leq \log (N_\tau(2,\dots,2))$. From Theorem 1.2 in  \cite{BDEG88}, $\nu_3$ is of pure type, atomic, or absolutely continuous.

We think that the same results can be proved with additional work for $k=4$, but for $k\geq 5$, other methods should be involved since the average contraction property seems to fail.
We were not able to find in the literature any general results allowing one to prove the identification of $\Theta_k$ with ${\sf Support}( \mu_k)$ or to compute the Hausdorff dimension of this support. We conjecture that for any $k\geq 1$, ${\sf Support}(\nu_k)$ coincides with $\Theta_k$ , and that the Lebesgue measure of this support (or of $\Theta_k$) is 0.

Now, we describe the distribution of the gaps sequence of the IBM thanks to $\nu_k$.
\begin{theo} \label{theo:multi-dim} Let $k$ be an integer larger than 0. If $(G_i,\iuk)$ is a random vector with distribution $\gamma_k$, then
\begin{equation}
(G_i,\iuk) \eqd (C_1 E_1,C_2 E_2,\dots,C_k E_k)
\end{equation}
where the $E_i$'s are i.i.d., $\Exp[1]$ distributed,  independent from $C[k]$, a random vector of law $\nu_k$.
\end{theo} 
According to this theorem and Proposition \ref{pro:ap}, we may deduce the following multivariate stochastic order bounds for $\gamma_k$, which somehow, describe the repulsive-attractive property of the gaps sequence.
\begin{pro}\label{eq:mso} Let $k$ be an integer larger than 0 and $(G_i,\iuk)$ a random vector with distribution $\gamma_k$. For any bounded increasing function $h:`R^k\to `R$
\[`E(h(E_i/(2k^2),1\leq i \leq k)) \leq `E(h(G_i,1\leq i \leq k))\leq`E(h(E_i/2,1\leq i \leq k))\] 
where the $E_i$ are i.i.d.  random variables $\Expo[1]$ distributed. 
\end{pro}
We can add here that the bound $2k^2$ is not tight (even in the case $k=2$ as one can see on Figure \ref{fig:IFS}).

\subsection{Iteration of reflected BM ad libitum}
\label{sec:IRMal}

In this section $X=|B|$ is the reflected BM (RBM), and $I^{(n)}=X_n\circ \dots \circ X_1$ the $n$th iterated RBM.
\begin{pro}\label{pro:rflbm} Let $t_0=0,t_1,\cdots,t_k$ be some non negative distinct real numbers. The sequence
\[\l(I^{(n)}(t_i), \izk\r) \dd (0,I_1,\dots,I_k)\]
where $(I_i, \iuk)$ is invariant by permutation and independent from the $t_i$'s and takes its value in $(0,+\infty)^k$. Moreover we have $I_1\sim \Expo[2]$. \par
Hence, the gaps sequence MC  $(\gaps{I^{(n)}(t_i), \izk},n\geq 1)$ converges and its limit $(G_i, \iuk):=\gaps{0,I_1,\dots,I_k}$ determine $(0,I_1,\dots,I_k)$:  for a uniform permutation $\tau\in {\cal S}\cro{1,k}$ independent from $(G_i, \iuk)$,
\[I[k]\eqd \l(\ov{G}_{\tau(i)}, \iuk\r).\]
\end{pro}

The proof of this proposition can be adapted from the proof of Theorem \ref{theo:ISP}. All these results rely on the ergodicity of the Markov chain $\l(I^{(n)}(t_i), \izk\r)$. The estimates needed to deal with the reflected Brownian case
can be simply adapted from the simple Brownian case.\par

We now describe the limiting gaps sequence using again the MC at the parameters level.

First, the Markov kernel of the iterated BM has a density. Set, for any $g,x,y\geq 0$, $M_{g}(x,y)=`P\l(B_{t+g}\in dy ~| B_t \in dx \r)$. We have, by André's reflection principle,
\ben\label{eq:MKBR}
M_g(x,y) = \l(\Phi_g(y-x)+ \Phi_g(y+x)\r)1_{y\geq 0}
\een
where $\Phi_g$ is the density of $B_g$.

The gap MC kernel can be described too adapting consideration of Section \ref{sec:BP} (in words, 0 stay at the left). Starting with some gaps sequence $\gaps{t_0=0,t_1,\dots,t_k}=g[k]$, we will have $\gaps{X(t_0)=0,X(t_1),\dots,X(t_k)}=x[k]$, if, for the same notation as in \eref{eq:S}, for $\iuk$, $X(\widehat{t}_i)= \ov{x}_{\tau(j)}$ for some permutation $\tau \in {\cal S}\cro{1,k}$ (instead of $\cro{0,k}$ for the iterated BM).
We then have in this case a solid link between the Markov kernel of the gaps sequence and of the initial chain, since $\wh{t}_j=\sum_{i=1}^j g_i$. We get in this case 
\ben \label{eq:Psi2}
\Psi_{g[k]}(x[k])&=&\sum_{\tau\in {\cal S}\cro{1,k}} \prod_{i=1}^k M_{g_i}\l(\ov{x}_{\tau(i-1)}, \ov{x}_{\tau(i)} \r) \,1_{x_i>0}.
\een
Therefore, modifying a bit \eref{eq:elt},  one can still see that $\MEX$ is stable by the MC with kernel $\Psi_k$. 
One observes using \eref{eq:sim-trans}, \eref{eq:double}, 
\[
\int_{`R^+{}^k} \prod_{i=1}^k \Exp[\lambda_i,g_i] \Psi_{g[k]}(x[k]) dg_1...dg_k ~~~~~~~~~~~~~~~~~~~~~~~~~~~~~~~~~~~~~~~~~~~~~~~~~~~~~~\]
\[~~~~~~~~~~~~~~~~~~~~~  =\frac{1}{2^k}\sum_{\tau\in {\cal S}\cro{1,k}} \prod_{i=1}^k \Exp\l[{\sqrt{2\lambda_i}}, |\ov{x}_{\tau(i)}-\ov{x}_{\tau(i-1)}|\r]+\Exp\l[{\sqrt{2\lambda_i}}, |\ov{x}_{\tau(i)}+\ov{x}_{\tau(i-1)}|\r]. \]
After expanding this product, one can again put together the elements ``containing a given'' $x_i$. Using the same considerations as those below Remark \ref{rem:1}, this is also 
\[=\frac{1}{2^k}\sum_{\tau\in {\cal S}\cro{1,k}}\sum_{D\subset\cro{1,k}}
 \prod_{i \in D} \Exp\l[{\sqrt{2\lambda_i}}, |\ov{x}_{\tau(i)}-\ov{x}_{\tau(i-1)}|\r]\prod_{i \in \complement D}\Exp\l[{\sqrt{2\lambda_i}}, |\ov{x}_{\tau(i)}+\ov{x}_{\tau(i-1)}|\r]. \]
This formula is the analogous in the case of RBM to that on the BM, \eref{eq:elt}. 
Let $E_{\tau,i}$ as defined in Section \ref{sec:BP}, and 
\be
E'_{\tau,i}&=& \{j: x_i\in | \ov{x}_{\tau(j)}|\} = \{j: \tau(j)\geq i\}\\
E''_{\tau,i}&=& \{j: x_i\in | \ov{x}_{\tau(j-1)}|\} = \{j: \tau(j-1)\geq i\}.
\ee
Set
\[F_{\tau,D,i}(c[k])=\sum_{j \in D, j \in E_{\tau,i}} \sqrt{2c_j} +\sum_{j \in \complement D,  j \in E'_{\tau,i}} \sqrt{2c_j}+  \sum_{j \in \complement D,  j \in E''_{\tau,i}} \sqrt{2c_j}\]
and 
\[w_{\tau,D}(c[k])=\frac{1}{2^k} \prod_{i=1}^k \frac{\sqrt{2c_i}}{F_{\tau,D,i}(c[k])}.\]
Again let 
\[F_{\tau,D}=(F_{\tau,D,i}, 1\leq i \leq k).\]
Similarly to Lemma \ref{lem:transfo} we have
\begin{lem}\label{lem:transfo2}If $f$ is the function $f(x[k])=\prod_{i=1}^k \Exp[\lambda_i,x_i]$  for some $\lambda[k]\in(0,+\infty)^k$, then 
  \begin{displaymath}
    \Op_k(f)(x[k]) = \sum_{D\subset\cro{1,k}} \sum_{\tau\in {\cal S}\cro{0,k}} w_{\tau,D}(\lambda[k]) \prod_{i=1}^k \Exp\l[ F_{\tau,D,i}(\lambda),x_i\r].  \end{displaymath}
\end{lem}
As in the iterated Brownian motion case, to this operator one can associate a Markov chain $Z^{(n)}$ with kernel $Q$ (defined as \eref{eq:parameter-kernel}) at the level of the parameters (see also \eref{eq:MC-parameter}). Again, the Markov chain $Z^{(n)}$ stays eventually confined in a compact region of $`R^+{}^k$ (the compact $[2,18k^2]^k$ is conserved by each of the $F_{\tau,D}$). By the same considerations as that of Section \ref{sec:BP}, Proposition \ref{pro:ap} holds for the present case (with $[2,18k^2]^k$ instead of $[2,2k^2]^k$).
The analogous of Theorem \ref{theo:multi-dim} holds too, for $\nu_k$ the fixed point of $Q$, and Proposition \ref{eq:mso} too, with $18k^2$ instead of $2k^2$ (again $18k^2$ is not tight).

\subsection{$n$th iteration of the BM ad libitum}
\label{sec:IRMn}
In the literature, the standard iterated Brownian motion corresponds to our process $I^{(2)}$. 
It has been deeply studied. It permits to construct solutions to partial differential equations~\cite{Fun79}. Burdzy studied some of its sample paths properties~\cite{Bur93}. Lot of results have been obtained around its probabilistic and analytic properties, see~\cite{Bur93,Ber96,ES99,BK95,Xiao98,KL99} and the references therein. The $n$th IBM permits to construct solutions of differential equations~\cite{OB09}, but they are less studied, only~\cite{Ber96} mentioned that his result can be extended to $n$th IBM. As far as we are aware of, there are no result concerning some description of the finite dimensional distributions of this process. In the sequel, we show that our gaps point of view allows one to give (a non trivial) description of them., but sufficiently simple to make some exact computations for small values of $n$ and $k$.

Let $n\geq 1$ be fixed, as well as $(t_0=0,t_1,\dots,t_k)$ some distinct numbers. The aim of this part is to describe the distribution of $(I^{(n)}(t_i), \izk)$, where $I^{(n)}=B_n \circ\cdots \circ B_1$, where the $B_i$ are i.i.d. two sided BM.

We built our reflection on the considerations presented in Section \ref{sec:IPE}. Start with formula \eref{eq:tran} which expresses the encoding Markov chain kernel. Here, of course, $\Phi_g$ is the Gaussian density.
Again, by \eref{eq:sim-trans}
\ben\label{eq:fondqd}
\int \prod_{i=1}^k \Exp[\lambda_i,g_i]K_{g[k],\tau}((g'_1,\dots,g'_k),\tau'\circ\tau) dg_1...dg_k
= \frac{1}{2^k}\prod_{i=1}^k  \Exp\l[\sqrt{2\lambda_i},|\Delta \ov{g'}_{\tau'(i)}|\r].
\een
Thanks to \eref{eq:iden1}, we can again rewrite the right hand side to code  the evolution on the parameter space. 
Setting $m_i=\min\{\tau'_{i-1},\tau'_{i}\}$ and $M_i=\max\{\tau'_{i-1},\tau'_{i}\}$ we get
$|\Delta \ov{g'}_{\tau'_i}|=  g'_{1+m_i}+\cdots + g'_{M_i}$. 
Once again, collect the different contribution: set $E_j(\tau,\tau')=\{~i: j \in \cro{m_i +1,M_i}\}$ the set of indices $i$ such that $g_j'$ contributes to $|\Delta \ov{g'}_{\tau'(i)}|$. The RHS of \eref{eq:fondqd} rewrites
\[w_{\tau,\tau'}(\lambda[k])  \prod_{j=1}^k\Exp\l[ F_{\tau,\tau',j}(\lambda),  g'_j\r].\]
where 
\[w_{\tau,\tau'}(\lambda[k])= \frac{1}{2^k}\prod_{j=1}^k  \frac{\prod_{i \in E_j(\tau,\tau')} \sqrt{2\lambda_i}}{\sum_{i\in E_j(\tau,\tau')}\sqrt{2\lambda_i}} \textrm{ and } F_{\tau,\tau',j}(\lambda[k])=\sum_{i\in E_j(\tau,\tau')}\sqrt{2\lambda_i}.\]
Consider $\MEX'_k$ the set of measures that are mixtures of distribution on $`R^k\times {\cal S}\cro{0,k}$  of the type $\l(\prod_{i=1}^k \Exp[\lambda_i]\r)\times \delta_{\tau}$ where $\delta_\tau$ is a Dirac on a permutation $\tau$.
The previous considerations show that the kernel $K$ operates linearly on $\MEX'_k$. It sends  $\l(\prod_{i=1}^k \Exp[\lambda_i]\r)\times \delta_{\tau}$ on 
\[\sum_{\tau'} w_{\tau,\tau'}(\lambda[k]) \l(\prod_{i=1}^k\Exp\l[F_{\tau,\tau',j}(\lambda[k]),  g'_j\r]\r)\times \delta_{\tau'\circ \tau}.\]
This can again be written at the parameter level under the form of a time homogeneous MC on $`R^k \times {\cal S}\cro{0,k}$, which, starting at time 0 at position $(\lambda[k],\tau)$, takes at time 1, the value $((F_{\tau,\tau',j}(\lambda[k]), 1\leq j \leq k), \tau'\circ\tau)$ with probability $w_{\tau,\tau'}(\lambda[k])$ (for any $\tau'\in {\cal S}\cro{0,k})$.  Denote again by $Q$ the corresponding kernel.  

This explicit description allow computations for small values of $k$ and of $n$.
Recall at the beginning of Section \ref{sec:IPE} the decoding map $C^{-1}$. 
Finally denoting by $ `E_{\lambda[k],\tau}$ the expectation when the initial encoding distribution is $\l(\prod_{i=1}^n \Exp(\lambda_i)\r) \times \delta_\tau$, we find
\ben\label{eq:qfdqfd}
`E_{\lambda[k],\tau}(f(I^{(n)}[k]))&= &\sum_{\tau'} \int_{R^+{}^k}{}Q^{n}_{(\lambda[k],\tau)}([d\lambda'_1,\dots,d\lambda'_k) ],\tau'\circ \tau)\\
&\times &\int_{`R^k} f(C^{-1}(y[k],\tau'\circ \tau))\l( \prod_{i=1}^n \Exp[\lambda'_i,y_i]\r) dy_1...dy_k.
\een
The LHS appears to the Laplace transform of $f(I^{(n)}(t_1),\dots,I^{(n)}(t_k))$ with respect to the initial gaps sequence, and then it characterises the distribution. This is not a simple description, but we think that it is the simplest representation of the finite dimensional distribution of the iterated BM one can find.

\section{Stable processes iterated ad libitum}
\label{sec:ISPal}
\subsection{Main results}
In this section, we consider independent two sided-stable processes $X_1,X_2, \dots,$ with parameters $(\alpha,\sigma,r)$ as defined in Section \ref{sec:P}, and there successive iterations $I^{(n)}=X_n \circ \cdots \circ X_1$. In this section, $\Phi_g$ is no more the Gaussian density but the density of $X_1(g)$. \par

 Two sided stable processes possess independent and stationary increments, as well as a scaling property which makes their iterations very similar to that of BM (general Lévy processes seem more difficult to handle because of this lacking scaling property).
Here are the convergence results we get for iterated stable processes $I^{(n)}$, as described in Section \ref{sec:P}.
\begin{theo} \label{theo:ISP}
Assume $\sigma\in(0,+\infty)$. Take $k,n\geq 1$ and some non zero $t_1,\dots,t_k$. Set, 
\[I^{(n)}[k]:=(I^{(n)}(t_1),\cdots,I^{(n)}(t_k)).\]
\begin{enumerate}
\compact
\item When $\alpha \leq 1$ and any $r$, for any $t>0$,  $I^{(n)}(t)$ does not converge in distribution in $\mathbb{R}$.
\item When $1 < \alpha \leq 2$ and $|r|>1$ then $I^{(n)}(t_1)$ does not converge in distribution in $\mathbb{R}$
\item When $1 < \alpha \leq 2$ and $|r|<1$  the MC $I^{(n)}[k]$ converges in distribution. The limit distribution $\mu_k$ does not depend on the $t_i$'s and is then exchangeable. For  $I[k] \sim \mu_k$, the equality $I[k]\eqd (X(I_1),\cdots, X(I_k))$ holds.
 Moreover,  $(I_2-I_1,\dots, I_k-I_1)\sim \mu_{k-1}$\\
When $r=0$, under $\mu_1$, $I_1\eqd `e\prod_{i\geq 0} |X(1)^{(i)}|^{1/{\alpha^i}}$ where the $X(1)^{(i)}$'s are i.i.d. copies of $X(1)$ and $`e$ is an independent uniform random sign.
\end{enumerate}
\end{theo}
\begin{rem}\label{rem:comp}
If $X$ and $X'$ are two stable processes with parameters $(\alpha,1,0)$ and $(\alpha,\sigma,0)$ for some $\alpha \in(1,2]$ and $\sigma>0$, then  $X'\eqd \sigma X$. 
The successive iteration of $(\alpha,1,0)$ and $(\alpha,\sigma,0)$ stable processes and limits (if any) can be compared by a simple coupling, but the property fails when $r\not= 0$. 
\end{rem}
\begin{rem} 
Notice that since $\mu_{k+1}$ is exchangeable, the rank of $I_1$ in $I[k+1]$ is uniform. Therefore the (random) number of indices $\#\{j: 2\leq j \leq k+1, I_j-I_1 >0 \}$ is uniform in $\cro{0,k}$. In other words, if $I[k]\sim\mu_k$ the rank of 0 in the list $(0,I_1,\dots,I_k)$ is uniform.
\end{rem}

\begin{lem} When $1 < \alpha \leq 2$ and $|r|<1$, the MC $(\G^{(n)}[k], n \geq 1)$ converges in distribution, and the limit distribution does not depend on the initial non-zero state.
\end{lem}
\begin{proof} Take some gaps sequence $g[k]\in \mathbb{R}^\star{}^k$. They are the gaps sequence of some non zeros and distinct times $t_0, t_1,\dots,t_k$. Start with $I^{(0)}[k+1]=(t_0,\dots,t_k)$.  Since $I^{(n)}[k+1]$  converges in distribution to a limit independent from the $t_i$ (Theorem \ref{theo:ISP}(3)), then the gaps sequence MC $\G^{(n)}[k]:=\gaps{I^{(n)}[k]}$  too.  
\end{proof}
Consider the case $k=1$ and  $X$ a symmetric stable process for some $\alpha\in(1,2]$ and $r=0$. Assume that the gap at time 0 is distributed as $G$, at time 1, the gap will be $|X(G)|$.
Then from the equality $G \eqd |X(G)|\eqd G^{1/\alpha} |X(1)|$ we infer
\ben\label{eq:dpq}
G\eqd \prod_{i\geq 0} |X(1)^{(i)}|^{\frac1{\alpha^i}},
\een
where the $X(1)^{(i)}$ are i.i.d. copies of $X(1)$ (the complete argument can be adapted from Section \ref{sec:PT}). When $r\neq 0$ there are not any such simple formula.

Let $\phi$ be the density of $G$ as defined in \eref{eq:dpq}. We have $\Op_1(\phi)=\phi$, and $\int_{g\geq 0} \phi(g) \Psi_g(x)dg=\phi(x)$ is an identification between the densities of $|X(G)|$ and of $|G|$. 
Using that $\phi[c,g]:=\frac{\phi(g/c)}c$ is the distribution of $cG$, we get that  $\int_{g\geq 0} \phi[c,g] \Psi_g(x)dg$ is the density of $|X(cG)|\eqd c^{1/\alpha}|X(G)|\eqd c^{1/\alpha} G$, which density is $\phi[c^{1/\alpha},x]$. All of this can be summed up in
\ben\label{eq:mono-stable}
\int_{g\geq 0} \phi[c,g] \Psi_g(x)dg =\phi[c^{1/\alpha},x].
\een 
Let $\nu[c]$ the distribution whose density is $\phi[c,.]$. The map $\Op_1$ sends $\nu[c]$ onto $\nu[c^{1/\alpha}]$ and is a linear application on the set of mixtures of distributions $\nu[c]$. This property which in the Brownian case allowed us to prove that $\Op_k$ was linear on the set of mixtures of product of exponential distributions can not be extended here. The reason is that, to separate the variables in \eref{eq:elt}, we used \eref{eq:iden1}. This important property holds only for exponential distributions, and it turns out that in the stable case, product measures of the form $\prod_{i=1}^k \phi[c_i,x_i]$ are not sent by $\Op_k$ on mixtures of measures of the same kind. We were not able to find a family of measures on which $\Op_k$ would operates simply but the discovery of such a family would be an important step for the identification of the distribution of $I[k]$.

\subsection{Occupation measure in the stable case}
\label{sec:om}
As stated in Proposition \ref{pro:C-K}, Curien and Konstantopoulos \cite{C-K} obtained some information about the occupation measure of the iterated Brownian motion ad libitum. 
In the stable case, when convergence holds, the family of limiting distributions $\mu_k$ are consistent, and since, they correspond to distribution of exchangeable vectors, by Kolmogorov extension theorem together with de Finetti representation theorem, there exists a random measure $\mu$, so that for any $k\geq 0$, $\mu_k$ is the distribution of $(U_1,\dots,U_k)$ i.i.d. random variables taken under $\mu$ (this is explained in the Brownian case in \cite{C-K}).

The main tool used in \cite{C-K} to characterise the regularity of the density of the occupation measure is a paper by Pitt \cite{Pit} only available in the Gaussian case. We are not able for the moment to get a similar result in the stable case, and then we renounce to go on our research in this direction. In view of Figure \ref{Fig:simu}, we may expect that for some small parameters $\alpha$ in $(1,2]$ (close to 1), the density of the local time should be not positive on the range of its support. 
\begin{figure}[htbp]
\centerline{\includegraphics[width=6cm,height=4cm]{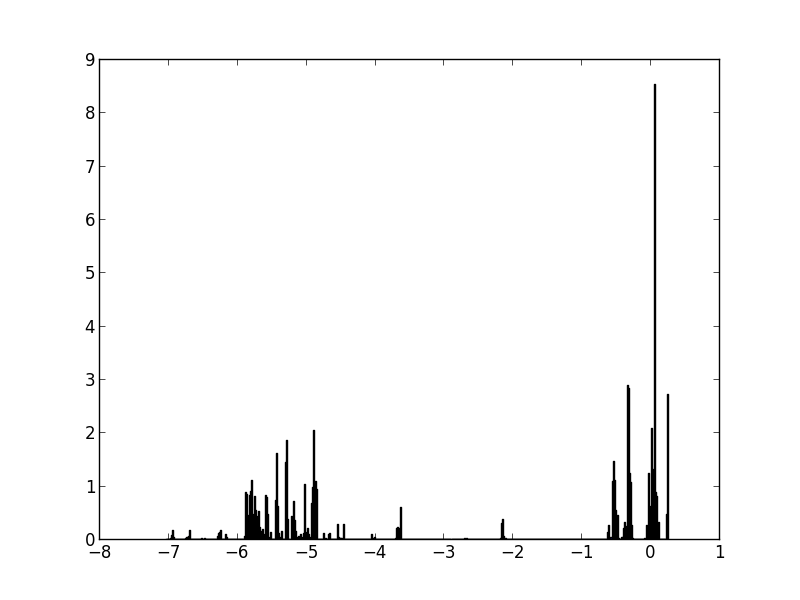}~~\includegraphics[width=6cm,height=4cm]{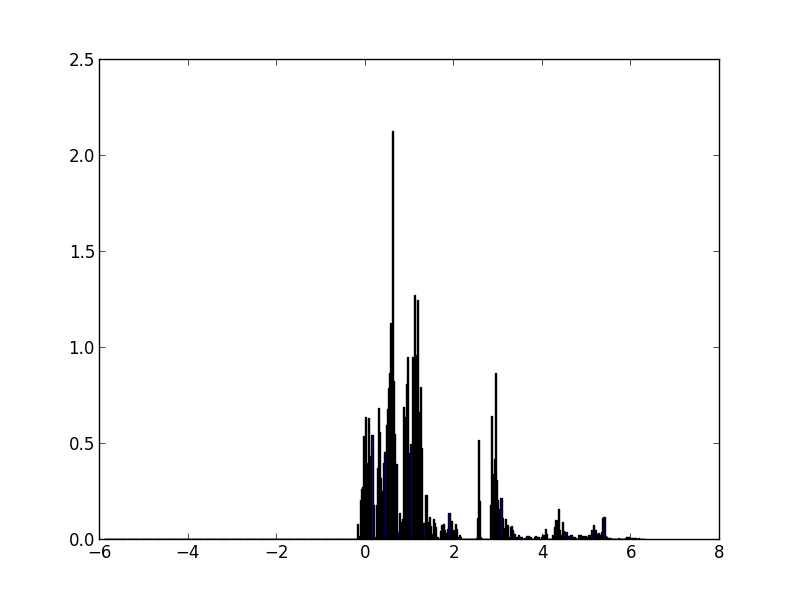}}
\centerline{\includegraphics[width=6cm,height=4cm]{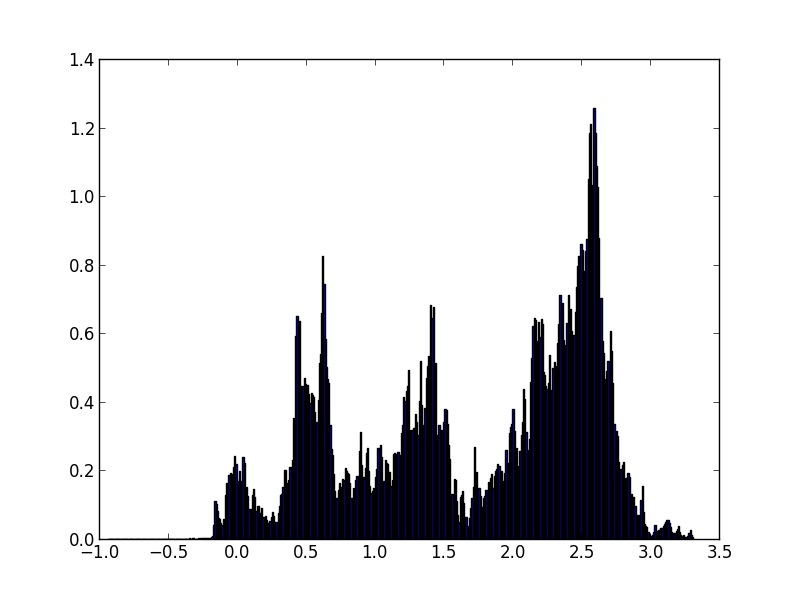}~~\includegraphics[width=6cm,height=4cm]{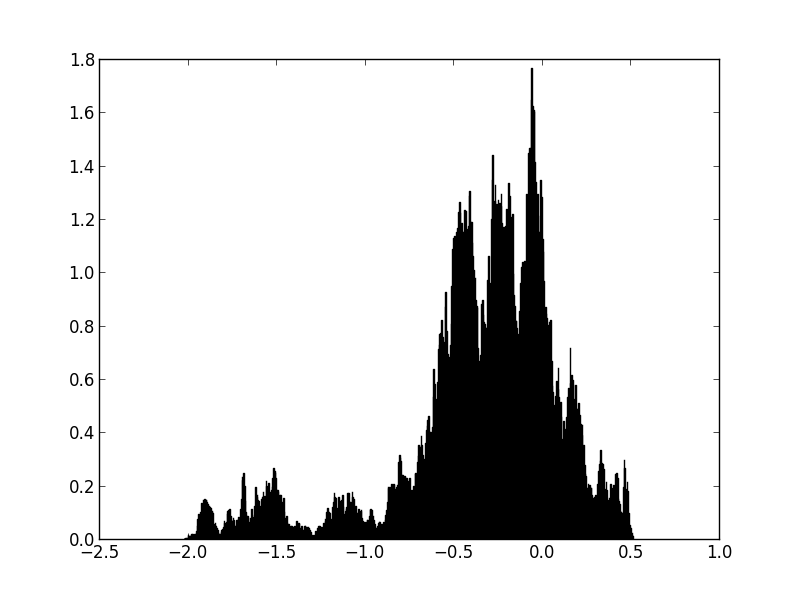}}
\caption{\label{Fig:simu} Simulation of the local time of the iterated centred stable processes ad libitum, in the case $\alpha=1.2, 1.5, 1.8$ and $2$. Each of them is made from an histogram made with a sample $(X^{(10)}(t_i),1\leq i \leq 10^6)$ starting from some fixed position.}
\end{figure}

In the next subsection, we discuss the finiteness of the support of the limiting occupation measure.
The proof follows the same structure as that of~\cite[Prop. 7]{C-K}. 
Let $P$ be any two-sided real process (in our case $P=X,I^{(n)} \text{ or } I$). The range of $P$ on  $[a,b]$ is defined by
\begin{equation}
R_P(a,b) = \sup_{a \leq t \leq b} P(t) - \inf_{a \leq t \leq b} P(t).
\end{equation} 
In the following, set  $D=R_X(0,1)$.

\begin{lem} \label{lem:rangeISP}
For any $|r|<1$ and $\alpha \in(1,2]$, for almost any $t \neq 0$, 
$R_{I^{(n)}}(0,t)$ converges in law to a r.v. $\Delta$ which does not depend on $t$.
Moreover, when $r = 0$, 
\begin{equation} \label{eq:rangeISPinf}
\Delta \eqd \prod_{i=0}^{\infty} D_i^{\alpha^{-i}}
\end{equation}
where the $D_i$'s are i.i.d. copies of $D$.
\end{lem}

\begin{proof}
Let $A_n(t) = \inf\{ I^{(n)}(u),0 \leq u \leq t\}$ and $B_n(t) = \sup\{ I^{(n)}(u),0 \leq u \leq t\}$.
 When $r=0$, 
\begin{eqnarray*} 
R_{I^{(n+1)}}(0,t)
& = & \sup_{A_n(t) \leq v \leq B_n(t)} X(v) - \inf_{A_n(t) \leq v \leq B_n(t)} X(v) \\
& = & R_X\left( A_n(t),B_n(t) \right) \label{eq:rangeISP} \\
& = & \left(B_n(t)-A_n(t)\right)^{\alpha^{-1}} D \\
& \eqd & \left(R_{I^{(n)}}(0,t)  \right)^{\alpha^{-1}} D.
\end{eqnarray*}
By iteration, we get
\begin{equation}
R_{I^{(n)}}(0,t) = t^{\alpha^{-n}} \prod_{i=1}^n D_i^{\alpha^{-(i+1)}}
\end{equation}
where $D_i$ are i.i.d. copies of $D$. Since $\alpha > 1$ and $t\neq 0$, $t^{\alpha^{-n}} \to 1$ when $n \to \infty$. Now, we have to prove the convergence in law of $\prod_{i=0}^{n-1} D_i^{\alpha^{-i}}$ as $n\to +\infty$. Write
\begin{displaymath}
\log \prod_{i=0}^{n-1} \left| D_i \right|^{\alpha^{-i}} = \sum_{i=0}^{n-1} \alpha^{-i} \log \left| D_i \right|.
\end{displaymath}\par
By the Doob's $\mathbb{L}^p$ inequality~\cite[Theorem II.1.7]{RY99}, for any $\beta \in \reals$,
\begin{equation}
\prob{D \geq x} \leq \prob{\sup_{0 \leq t \leq 1} |X(t)| \geq \frac{x}{2}} \leq \frac{2^\beta \esp{X(1)^\beta}}{x^\beta}.
\end{equation} 
But, $\esp{X(1)^\beta} < \infty$ if $\beta < \alpha$. For $\beta=1 < \alpha$, 
\[\prob{ \alpha^{-i} \log \left| D_i \right| > i^{-2}} \leq \frac{\text{Cste}}{e^{\alpha^{i}i^{-2}}},\]
which is a summable sequence, since $\alpha>1$. By Borel-Cantelli's lemma, $\prod_{i=1}^n D_i^{\alpha^{-i}}$ converges as $n\to+\infty$. This ends the proof when $r=0$.\\
In the general case,  write
\begin{eqnarray*}
R_{I^{(n+1)}}(0,t)  & = & R_X\left( A_n(t),B_n(t) \right) \\ 
& \leq & (B_n(t) - A_n(t))^{\alpha^{-1}} D + r (B_n(t) - A_n(t)) \\
& \eqd & (R_{I^{(n)}}(0,t))^{\alpha^{-1}} D + r R_{I^{(n)}}(0,t) \label{eq:inegr}
\end{eqnarray*}

To prove that $R_{I^{(n)}}(0,t)$ converges, we use Theorem~13.0.1 in~\cite{M-T}. By~\eqref{eq:inegr},
\begin{displaymath}
\esp{R_{I^{(n+1)}}(0,t)|R_{I^{(n)}}(0,t)} -  R_{I^{(n)}}(0,t) \leq (R_{I^{(n)}}(0,t))^{\alpha^{-1}} \esp{D} - (1-r) R_{I^{(n)}}(0,t).
\end{displaymath}
So if $R_{I^{(n)}}(0,t) > \displaystyle \left( \frac{\esp{D}}{1-r} \right)^{\frac{1}{1-\alpha^{-1}}} = M$, then $\esp{R_{I^{(n+1)}}(0,t) |R_{I^{(n)}}(0,t)} - R_{I^{(n)}}(0,t) \leq -1$; else $\esp{R_{I^{(n+1)}}(0,t)|R_{I^{(n)}}(0,t)} - R_{I^{(n)}}(0,t) \leq M^{\alpha^{-1}} \esp{D} +1 =b$, from what we deduce\begin{equation}
\esp{R_{I^{(n+1)}}(0,t)|R_{I^{(n)}}(0,t)} - R_{I^{(n)}}(0,t) \leq -1 + b 1_{[0,M]}(R_{I^{(n)}}(0,t)).
\end{equation}
This proves the ergodicity of $\left( R_{I^{(n)}}(0,t) ; n \geq 0 \right)$ by \cite[Theorem~13.0.1(iv)]{M-T}.
\end{proof}

By Lemma~\ref{lem:rangeISP} and the arguments of~\cite[Section 3.2]{C-K}, this proves that $\phi$ has a bounded support a.s.

\subsection{Proofs of Theorem \ref{theo:ISP}}
\label{sec:PT}
The main technical point (Theorem \ref{theo:ISP} (3)) concerns the convergence of the MC $(I^{(n)}(t_i), \izk)$ in the stable case from which we will derive the other convergence theorem of the paper by some slight modifications.

In the proof $\widetilde{X}(1)$ stands for the symmetric part of $X(1)$ so that $X(t)\eqd rt+|t|^{1/\alpha}\widetilde{X}(1)$.\\
1. Assume $\alpha <1$, and $r \in \mathbb{R}$. One has $I^{(n)}(t)\eqd rI^{(n-1)}(t) + |I^{(n-1)}(t)|^{1/\alpha} \widetilde{X}(1)$. Since $1/\alpha>1$, it is apparent that $|I^{(n)}(t)|$ should become very large. To prove this, we compare $I^{(n)}$ with a deterministic geometric sequence $c^n$ for $(1/\alpha)> c>1$.
\[`P\l(|I^{(n)}(t)|\geq c^n \,|\, |I^{(n-1)}(t)| \geq c^{n-1}\r)\geq \inf_{x \geq c^{n-1}} `P( |rx+x^{1/\alpha} \widetilde{X}(1)| \geq c^n)\]
For any $x\geq c^{n-1}$, 
\be
 `P( |rx+x^{1/\alpha} \widetilde{X}(1)| \geq c^n)&=& 1- `P\l(\frac{- c^n-rx}{x^{1/\alpha}}\leq  \widetilde{X}(1)\leq  \frac{c^n-rx}{x^{1/\alpha}}\r)
\ee
and since stable distribution possesses continuous density $h$ at 0 (see Feller \cite[sec. XV(3)]{fel2}), this is
 \be
& \geq & 1- C h(0) \l(\frac{c^n-rx}{x^{1/\alpha}} -\frac{- c^n-rx}{x^{1/\alpha}}\r)\\
& = & 1- C h(0) \l(\frac{2c^n}{x^{1/\alpha}}\r)\\
& \geq &  1- C h(0) \l(\frac{2c^n}{c^{(n-1)/\alpha}}\r) 
\ee
for $n$ large enough and some constant $C>0$. We deduce
\[`P(|I^{(n)}(t)|\geq c^n ,\forall n\geq 1 )>0.\]

 When $\alpha=1$, for any $r$, $I^{(n)}(t)\eqd r I^{(n-1)}(t) + |I^{(n-1)}(t)| \widetilde{X}_1^{(n-1)}$. As the distribution of $\widetilde{X}_1$ is symmetric with respect to $0$, $(r I^{(n-1)}(t),|I^{(n-1)}(t)| \widetilde{X}_1^{(n-1)}) \eqd (rI^{(n-1)}(t),I^{(n-1)}(t) \widetilde{X}_1^{(n-1)})$, from which we get $I^{(n)}(t) \eqd I^{(n-1)}(t) (r + \widetilde{X}_1^{(n-1)} ) \eqd t \prod_{i=1}^n \left( r + \widetilde{X}_1^{(i)} \right)$.  Taking the logarithm, one sees that $I^{(n)}(t)$ does not converge in distribution.\\

2. The proof we provide here is valid for any $\alpha>0$. In the sequel we assume $r>1$ (the case $r<-1$ is similar).  For a fixed $t$, $I^{(n)}(t)\eqd rI^{(n-1)}(t) + |I^{(n-1)}(t)|^{1/\alpha} \widetilde{X}(1)$. For $r>1$, $I^{(n)}(t)$ can be compared with a geometric sequence with common ratio $s\in(1,r)$. Write
\[`P\l(|I^{(n)}(t)|\geq s^n \,|\, |I^{(n-1)}(t)| \geq s^{n-1}\r)\geq \inf_{x \geq s^{n-1}} `P( |rx+x^{1/\alpha} \widetilde{X}(1)| \geq s^n)\]
For any $x\geq s^{n-1}$, 
\be
 `P( |rx+x^{1/\alpha} \widetilde{X}(1)| \geq s^n)&=& 1- `P(- s^n\leq rx+x^{1/\alpha} \widetilde{X}(1)\leq  s^n)\\
&\geq & 1- `P( rx+x^{1/\alpha} \widetilde{X}(1)\leq  s^n)\\
&= & 1- `P(\widetilde{X}(1)\geq \frac{ rx-s^n}{x^{1/\alpha}})\\
&= & 1- `P(\widetilde{X}(1)\geq \frac{(r-s)x+ sx-s^n}{x^{1/\alpha}})\\
&\geq & 1- `P(\widetilde{X}(1)\geq \frac{(r-s)x}{x^{1/\alpha}})\\
&= & 1- `P(\widetilde{X}(1)\geq (r-s)x^{1-1/\alpha})\\
&\geq &1- c s^{(n-1)(1-\alpha)}
\ee
for $n$ large enough (we have use that  $`P(\widetilde{X}(1) \geq v) \leq c' v^{-\alpha}$ for some $c'$ and $v\geq s$, and that $r-s$ is a constant, and the symmetry of the distribution of $\widetilde{X}(1)$). We deduce from that that 
\[`P(|I^{(n)}(t)|\geq s^n ,\forall n\geq 1 )>0.\]
\par

3. The proof  of the convergence of $I^{(n)}[k]$ we propose is adapted from Curien \& Konstantopoulos \cite{C-K}. 

The sequence $(I^{(n)}[k],n\geq 1)$ is a MC, and its Markov kernel is given by
\[P(y[k];A)=`P((X(y_1),\dots,X(y_k))\in A),\]
for any $y[k]\in `R^k$, any Borelian $A\in \mathbb{R}^k$. 
As in \cite{C-K}, the Markov chain $I^{(n)}[k]$ is aperiodic, and irreducible with respect to the $p$-dimensional Lebesgue measure on $`R^k$.
We prove that it is Harris recurrent (and then possesses a unique invariant distribution), following the elements that can be found in Section 5.5.1. Meyn \& Tweedie \cite{M-T}.
Set for any $M>0$
 \[S_M=\{x[k]\in `R^k, M^{-1}\leq |x_i|\leq M,|x_i-x_j|\geq M^{-1}\}.\]
Denote by $f_{x[k]}$ the density of $(X(x_1),\dots,X(x_k))$, and let 
\[F_M(z[k])=\min_{x[k]\in S_M} f_{x[k]}(z[k]).\]
It is easily seen that $F_M$ is the density of a $\sigma$-finite measure $\mu_M$ on $`R^k$, with total mass $c_M=\int_{R^k}F_M(z[k])dz_1\cdots dz_k>0$ and satisfy $F_M(z[k])>0$ for any $z_1,\dots,z_k$. This provides the following bounds on the Markov kernel of our MC:
\[`P((X(x_1),\dots,X(x_k))\in A)\geq \mu_M(A),~~~ \textrm{for all }x[k]\in S_M.\]
This is the minoration condition (5.2) in  \cite{M-T}: the set $S_M$ is $\mu_M$-petite. To prove the Harris recurrence of the MC, it suffices to prove that for some $M>0$, the expected hitting time of $S_M$ by $I^{(n)}[k]$ starting from $x[k]$, is bounded above for $x[k] \in S_M$. 
Consider, for $x[k]\in `R^{+}{}^k$
 \[V(x[k])=U(x[k])+G(x[k])\]
with $U(x[k])=\max\{|x_i|, i \in\cro{1,k}\}$, $G(x[k])=\sum_{0\leq i <j\leq k}|x_i-x_j|^{-1/\alpha}$ (where $x_0=0$).
The potential function $V$ is unbounded on $`R^k$, and its drift is defined by
\[DV(x[k]):=PV(x[k])-V(x[k])=`E(V(X(x_1),\dots,X(x_k)))-V(x[k]),~~\textrm{ for } x[k]\in `R^k.\]
We just have to prove that
\ben\label{eq:qsdd}
\Delta V(x[k]) \leq - a +b 1_{S_M}(x),~~ x[k]\in `R^k.
\een

We have for any $\lambda>0$,
\be
PU(x[k])& = &`E\l(\max_{i \in\cro{1,k}}|X_{x_i}|\r)
      = `E[\max |\bar{X_{x_i}}|+|r| \,|x_i|] \leq |r|U(x)+ \lambda^{1/\alpha}`E[\max |\widetilde{X}_{x_i/\lambda}|]
\ee
and then taking $\lambda=U(x[k])$, we get
\be
PU(x[k])&\leq& |r|U(x[k])+ U(x[k])^{1/\alpha} C_1
\ee
where $C_1= `E[\max_{-1\leq s \leq 1} |X_{s}|]$. 
Now,
\be
PG(x[k])&=&\sum_{0\leq i <j\leq k}`E \l[|X_{x_i}-X_{x_j}|^{-1/\alpha}\r]\\
     &=&\sum_{0\leq i <j\leq k}`E \l[|\widetilde{X}_{x_i-x_j}+r(x_{i}-x_j)|^{-1/\alpha}\r]
\ee
We decompose each term in the sum using
\be
\frac{1}{|\widetilde{X}_{x_i-x_j}+r(x_{i}-x_j)|^{1/\alpha}}&= &\frac{1_{{\sf Sign}(\widetilde{X}_{x_i-x_j})={\sf Sign}(r)}+1_{{\sf Sign}(\widetilde{X}_{x_i-x_j})\neq {\sf Sign}(r)}}{|\widetilde{X}_{x_i-x_j}+r(x_{i}-x_j)|^{1/\alpha}}\\
&\leq & \frac{1_{{\sf Sign}(\widetilde{X}_{x_i-x_j})={\sf Sign}(r)}}{|\widetilde{X}_{x_i-x_j}|^{1/\alpha}}+\frac{1_{{\sf Sign}(\widetilde{X}_{x_i-x_j})\neq {\sf Sign}(r)}}{||\widetilde{X}_{x_i-x_j}|-|r(x_{i}-x_j)||^{1/\alpha}}
 \ee
By symmetry and unimodality of the density of centred stable distributions, one has
\[`E \l[|\widetilde{X}_{x_i-x_j}+r(x_{i}-x_j)|^{-1/\alpha}\r]\leq 2 `E\l({|\widetilde{X}_{x_i-x_j}|^{-1/\alpha}}\r)\]
Hence
\be
PG(x[k])    &=&2\sum_{0\leq i <j\leq k} (|x_i-x_j|^{1/\alpha})^{-1/\alpha}`E(|\widetilde{X}_1|^{-1/\alpha})\\
     &\leq & 2(k^2)^{1-1/\alpha}`E(|\widetilde{X}_1|^{-1/\alpha}) G(x)^{1/\alpha}
\ee
this last inequality come from $\sum_{i=1}^m |y_i|^{-1/\alpha^2} \leq m^{1-1/\alpha}\l(\sum_{i=1}^m |y_i|^{-1/\alpha}\r)^{1/\alpha}$ which can be viewed as an application of Jensen inequality: take $W$ uniform in $\cro{1,m}$, $f(x)=x^{1/\alpha}$. Since $f$ is concave
 $`E(f(|y_W|^{-1/\alpha}))\leq f(`E(|y_W|^{-1/\alpha}))$ which is equivalent to $\frac{1}{m}\sum_{i=1}^m|y_i|^{-1/\alpha^2} \leq (\frac{1}{m}\sum_{i=1}^m |y_i|^{-1/\alpha})^{1/\alpha}$.
We get, by convexity, for some constant $C_k$ and $C'_k$,
\be
PV(x[k])=PU(x[k])+PG(x[k])&\leq& C_k(U(x[k])^{1/\alpha}+G(x[k])^{1/\alpha}) +|r|U(x[k])\\
                 &\leq& C'_k V(x[k])^{1/\alpha}+|r|V(x[k])
\ee
which implies 
\ben\label{eq:dqs}
\Delta V(x[k])
&\leq& C'_kV(x[k])^{1/\alpha}-(1-|r|)V(x[k]).
\een
If $x[k]\notin S_M$ then there exists $i$ such that $|x_i|\geq M$ or $(i,j)$ such that $|x_i-x_j|\leq 1/M$. In the first case $V(x[k])\geq M$ and in the second one, $V(x[k])\geq M^{1/\alpha}$. For $M\geq 1$, we thus have $V(x[k])\geq M^{1/\alpha}$ for $x\notin S_M$. The RHS of \eref{eq:dqs} rewrites
  $V(x[k])^{1/\alpha}(C'_k-(1-|r|)V(x[k])^{1-1/\alpha})$. For $M$ chosen such that $(1-|r|)(M^{1/\alpha})^{1-1/\alpha}\geq \max\{1,2C'_k\}$, 
$V(x[k])^{1/\alpha}(C'_k-(1-|r|)V(x)^{1-1/\alpha})\leq -C'_k V(x[k])^{1/\alpha} \leq -C'_k M^{1/\alpha^2}$. 

For $x[k]\in  S_M$, $0\leq V(x[k]) \leq M+(k+1)^2M^{1/\alpha}<+\infty$ and then 
 $\Delta V(x[k])$ is bounded on $S_M$, this allows one to prove that \eref{eq:qsdd} holds for $C=S_M$ and $M$ large enough.

To end the proof, we need to prove the exchangeability of $I[k]$. The argument is general, and present in \cite{C-K}. 
Take $\sigma \in {\cal S}\cro{1,k}$ and $t_1,\dots,t_k$ distinct and non zeros. By the proof above, both $(I^{(n)}(t_i),\iuk)$ and $(I^{(n)}(t_{\sigma(i)}),\iuk)$ converge to $I[k]$. 
So, $(I^{(n)}(t_{\sigma(i)}),\iuk)$ converges to $I[k]$ and to $(I_{\sigma(i)},\iuk)$.
Hence, $I[k] \eqd (I_{\sigma^{}(i)},\iuk)$ for any $\sigma$. ~ $\Box$

\section{Conclusion}
\label{sec:Con}
In the paper, we have presented some results and some tools allowing to study iterated independent processes. Our tools are really useful only for processes with increments independent and stationary. Hence, the global frame is that of Lévy processes. But what we did for stationary process could probably done for continuous MC, homogeneous or not. For example it is likely that one can get some results on iterated Ornstein-Uhlenbeck processes whose increment are simple enough to be controlled.

\small
\bibliographystyle{abbrv}

\end{document}